\documentclass[reqno,12pt]{amsart}
\usepackage{times}

\newtheorem{thm}{Theorem}
\newtheorem{lemma}[thm]{Lemma}
\newtheorem{cor}[thm]{Corollary}

\newtheorem*{thma}{Theorem A}
\newtheorem*{thmb}{Theorem B}

\newcommand{\D}{{\mathbb D}}

\newcommand{\C}{{\mathbb C}}
\newcommand{\bloch}{{\mathcal B}}
\newcommand{\cn}{\C^n}
\newcommand{\bn}{{\mathbb B}_n}
\newcommand{\inb}{\int_{\bn}}
\newcommand{\ind}{\int_\D}

\begin{document}

\title[Integral Representation]
{An Integral Representation for\\ Besov and Lipschitz Spaces}

\author{Kehe Zhu}
\address{Department of Mathematics and Statistics, State University
of New York, Albany, NY 12222, USA}
\email{kzhu@math.albany.edu}


\begin{abstract}
It is well known that functions in the analytic Besov space $B_1$ on the unit disk $\D$
admits an integral representation
$$f(z)=\ind\frac{z-w}{1-z\overline w}\,d\mu(w),$$
where $\mu$ is a complex Borel measure with $|\mu|(\D)<\infty$. We generalize this
result to all Besov spaces $B_p$ with $0<p\le1$ and all Lipschitz
spaces $\Lambda_t$ with $t>1$. We also obtain a version for Bergman and Fock spaces.
\end{abstract}

\maketitle

\section{Introduction}

Let $\D$ denote the unit disk in the complex plane $\C$, $H(\D)$ denote the space of 
all analytic functions in $\D$, and $dA$ denote the normalized area measure on $\D$. 
For $0<p<\infty$ we consider the analytic Besov space $B_p$ consisting of functions 
$f\in H(\D)$ with the property that $(1-|z|^2)^kf^{(k)}(z)$ belongs to $L^p(\D,d\lambda)$, 
where
$$d\lambda(z)=\frac{dA(z)}{(1-|z|^2)^2}$$
is the M\"obius invariant area measure on $\D$ and $k$ is any positive integer
such that $pk>1$. The space $B_p$ is independent of the integer $k$ used.

It is well known that an analytic function $f$ in $\D$ belongs to $B_1$ if and only
if there exists a complex Borel measure $\mu$ on the $\D$ such that $|\mu|(\D)<\infty$ and
\begin{equation}
f(z)=\ind\frac{z-w}{1-z\overline w}\,d\mu(w),\qquad z\in\D.
\label{eq1}
\end{equation}
See \cite{AF,AFP,Z1}. The purpose of this paper is to generalize the above result to 
several other spaces, including Besov spaces, Lipschitz spaces, Bergman 
spaces, and Fock spaces. We state our main results as Theorems A and B below.

\begin{thma}
Suppose $0<p\le1$, $0<r<1$, and $f$ is analytic in $\D$. Then $f\in B_p$ if and
only if it admits a representation
$$f(z)=\ind\frac{z-w}{1-z\overline w}\,d\mu(w),\qquad z\in\D,$$
where $\mu$ is a complex Borel measure on $\D$ such that the localized function
$z\mapsto|\mu|(D(z,r))$ belongs to $L^p(\D,d\lambda)$, where
$$D(z,r)=\left\{w\in\D:\left|\frac{z-w}{1-z\overline w}\right|<r\right\}$$
is the pseudo-hyperbolic disk at $z$ with radius $r$.
\label{1}
\end{thma}

Recall that for any real number $t$, the analytic Lipschitz space $\Lambda_t$ on the unit
disk consists of functions $f\in H(\D)$ such that $(1-|z|^2)^{k-t}f^{(k)}(z)$ is bounded,
where $k$ is any nonnegative integer greater than $t$.

\begin{thmb}
Suppose $t>1$, $0<r<1$, and $f$ is analytic in $\D$. Then $f\in\Lambda_t$ if and only if
$$f(z)=\ind\frac{z-w}{1-z\overline w}\,d\mu(w),\qquad z\in\D,$$
for some complex Borel measure $\mu$ with the property that
$$\sup_{z\in\D}\frac{|\mu|(D(z,r))}{(1-|z|^2)^t}<\infty.$$
\end{thmb}

In addition to Besov and Lipschitz spaces in dimension $1$, where the integral representation
looks particularly nice, we will also consider Bergman and Fock spaces in higher dimensions.

\section{Preliminaries on measures}

Suppose $\mu$ is a complex Borel measure on $\D$ and $r\in(0,1)$. We can define two
functions on $\D$ as follows.
$$\mu_r(z)=\mu(D(z,r)),\qquad \widehat\mu_r(z)=\frac{\mu(D(z,r))}{(1-|z|^2)^2}.$$
It is well known that the area of the pseudo-hyperbolic disk $D(z,r)$ is
$$\pi r^2\left(\frac{1-|z|^2}{1-r^2|z|^2}\right)^2,$$
which is comparable to $(1-|z|^2)^2$ whenever $r$ is fixed. That is why we think of
$\widehat\mu_r$ as an averaging function for the measure $\mu$. We will call $\mu_r$
a localized function for $\mu$. The behavior of $\mu_r$ and $\widehat\mu_r$ is often 
independent of the particular radius $r$ being used.

Another averaging function for $\mu$ is the so-called Berezin transform of $\mu$. We need
the assumption $|\mu|(\D)<\infty$ in order to define the Berezin transform:
$$\widetilde\mu(z)=\ind\frac{(1-|z|^2)^2}{|1-z\overline w|^4}\,d\mu(w),\qquad z\in\D.$$
See \cite{Z1} for basic information about these averaging operations.

We will need to decompose the unit disk into roughly equal-sized parts in the 
pseudo-hyperbolic metric. More specifically, we will need the following result.

\begin{lemma}
For any $0<r<1$ there exists a sequence $\{z_n\}$ in $\D$ and 
a sequence of Borel subsets $\{D_n\}$ of $\D$ with the following properties:
\begin{enumerate}
\item[(a)] $\D=D_1\cup D_2\cup\cdots\cup D_n\cup\cdots$.
\item[(b)] The sets $D_n$ are mutually disjoint.
\item[(c)] $D(z_n,r/4)\subset D_n\subset D(z_n,r)$ for every $n$.
\end{enumerate}
\label{2}
\end{lemma}

\begin{proof}
This is well known. See \cite{Z1} for example.
\end{proof}

Any sequence $\{z_n\}$ satisfying the three conditions above will be called an $r$-lattice
in the pseudo-hyperbolic metric.

\begin{lemma}
Suppose $\mu$ is a positive Borel measure on $\D$ and $0<p\le\infty$. If $r$ and $s$
are two radii in $(0,1)$ and $\{z_n\}$ is an $r$-lattice in the pseudo-hyperbolic metric. 
Then the following conditions are equivalent.
\begin{enumerate}
\item[(a)] The function $\widehat\mu_s(z)$ belongs to $L^p(\D,d\lambda)$.
\item[(b)] The sequence $\widehat\mu_r(z_n)$ belongs to $l^p$.
\end{enumerate}
If $1/2<p\le\infty$, then the above conditions are also equivalent to
\begin{enumerate}
\item[(c)] The function $\widetilde\mu(z)$ belongs to $L^p(\D,d\lambda)$.
\end{enumerate}
\label{3}
\end{lemma}

\begin{proof}
This is also well known. See \cite{Z1} for example.
\end{proof}

As a consequence of the above lemma on the averaging function $\widehat\mu_r$ we obtain
several equivalent conditions for the localized function $\mu_r$.

\begin{cor}
Suppose $0<p\le\infty$ and $\mu$ is a positive Borel measure on $\D$. If $r$ and $s$
are two radii in $(0,1)$ and $\{z_n\}$ is an $r$-lattice in the pseudo-hyperbolic 
metric. Then the following conditions are equivalent.
\begin{enumerate}
\item[(a)] The sequence $\{\mu_r(z_n)\}$ belongs to $l^p$.
\item[(b)] The function $\mu_s(z)$ belongs to $L^p(\D,d\lambda)$.
\end{enumerate}
If $p\not=\infty$, the above conditions are also equivalent to
\begin{enumerate}
\item[(c)] The function $\widehat\mu_s(z)$ belongs to $L^p(\D,dA_{2(p-1)})$, where
$$dA_{2(p-1)}(z)=(1-|z|^2)^{2(p-1)}\,dA(z).$$
\end{enumerate}
\label{4}
\end{cor}

\begin{proof}
Consider the positive Borel measure
$$d\nu(z)=(1-|z|^2)^2\,d\mu(z).$$
It is well known that $(1-|z|^2)^2$ is comparable to $(1-|z_n|^2)^2$ for $z$ in $D(z_n,t)$,
where $t$ is any fixed radius in $(0,1)$. See \cite{Z1} for example. Thus $\widehat\nu_t(z)$ 
is comparable to $\mu_t(z)$, and $\widehat\nu_t(z)$ is also comparable to 
$(1-|z|^2)^2\widehat\mu_t(z)$. The desired result then follows from Lemma~\ref{3}.
\end{proof}

In view of the equivalence of conditions (a) and (c) in Lemma~\ref{3}, it is tempting 
to conjecture that condition (c) in Corollary~\ref{4} above is equivalent to
$\widetilde\mu\in L^p(\D,dA_{2(p-1)})$ whenever $p>1/4$. It turns out that this is not
true. This already fails at $p=1$. In fact, if $p=1$, the condition $\widetilde\mu\in
L^p(\D,dA_{2(p-1)})$ means
\begin{eqnarray*}
+\infty&>&\ind\widetilde\mu(z)\,dA(z)\\
&=&\ind(1-|z|^2)^2\,dA(z)\ind\frac{d\mu(w)}{|1-z\overline w|^4}\\
&=&\ind\,d\mu(w)\ind\frac{(1-|z|^2)^2\,dA(z)}{|1-z\overline w|^4}.
\end{eqnarray*}
This together with Lemma 3.10 in \cite{Z1} shows that, for $p=1$, the condition
$\widetilde\mu\in L^p(\D,dA_{2(p-1)})$ is the same as
$$\ind\log\frac1{1-|w|^2}\,d\mu(w)<\infty.$$
On the other hand, it is easy to see that condition (a) in Corollary~\ref{4}, for $p=1$,
simply means $\mu(\D)<\infty$, which is obviously different from the integral condition above.

The Berezin transform will not really be used in the rest of the paper, but it is always
interesting and insightful to compare the behavior of $\widehat\mu_r$ and $\widetilde\mu$.

Another notion critical to the integral representation of Lipschitz spaces is that of
Carleson measures. 

Let $t>0$. We say that a positive Borel measure $\mu$ on the unit disk $\D$ is $t$-Carleson if
$$\sup_{z\in\D}\frac{\mu(D(z,r))}{(1-|z|^2)^t}<\infty$$
for some $r\in(0,1)$. It is well known that if the above condition holds for some $r\in(0,1)$,
then it holds for every $r\in(0,1)$. Thus being $t$-Carleson is independent of the radius
$r$ used in the definition.

If $t>1$, then every $t$-Carleson measure is finite. In fact, in this case, we use 
Lemma~\ref{2} to get
\begin{eqnarray*}
\mu(\D)&=&\sum_{n=1}^\infty\mu(D_n)\le\sum_{n=1}^\infty\mu(D(z_n,r))\\
&\le&C\sum_{n=1}^\infty(1-|z_n|^2)^t\le C'\sum_{n=1}^\infty\int_{D_n}(1-|z|^2)^{t-2}\,dA(z)\\
&=&C'\ind(1-|z|^2)^{t-2}\,dA(z)<\infty.
\end{eqnarray*}
It is clear from the above argument that not every $t$-Carleson measure is finite when
$t\le1$.

If $t>1$, then $\mu$ is $t$-Carelson if and only if for some (or every) $0<p<\infty$ there
exists a constant $C=C_p>0$ such that
$$\ind|f(z)|^p\,d\mu(z)\le C\ind|f(z)|^p(1-|z|^2)^{t-2}\,dA(z)$$
for all $f\in H(\D)$. Because of this, such measures are also called Carleson measures for
Bergman spaces. When $t>1$, it is also known that $\mu$ is $t$-Carleson if and only if
there is a constant $C>0$ such that $\mu(S_h)\le Ch^t$ for all ``Carleson squares'' $S_h$
of side length $h$. See \cite{Z2,RZ}.

We warn the reader that there is a fine distinction between $1$-Carleson measures defined 
above and the classical Carleson measures (for Hardy spaces). This can be seen by
considering an arbitrary Bloch function $f$ in the unit disk. In fact, for such a function
if we define the measure $\mu$ by
$$d\mu(z)=(1-|z|^2)|f'(z)|^2\,dA(z),$$
then $\mu$ is $1$-Carleson, because
\begin{eqnarray*}
\mu(D(z,r))&=&\int_{D(z,r)}(1-|w|^2)|f'(w)|^2\,dA(w)\\
&\sim&\frac1{1-|z|^2}\int_{D(z,r)}(1-|w|^2)^2|f'(w)|^2\,dA(w)\\
&\le& C(1-|z|^2).
\end{eqnarray*}
But $\mu$ is not a classical Carleson measure, because being so would mean that $f$ is in
BMOA. See \cite{G}. It is certainly well known that BMOA is strictly contained in the 
Bloch space. A classical Carleson measure is $1$-Carleson, but not the other way around.

Similarly, for $0<t<1$, there is a subtle difference between measures satisfying the condition
$\mu(S_h)\le Ch^t$ and those satisfying the condition $\mu(D(z,r))\le C(1-|z|^2)^t$.

\section{Besov spaces in the unit disk}

We prove Theorem A in this section. For this purpose we need to make use of weighted
Bergman spaces. Thus for any $\alpha>-1$ we let
$$dA_\alpha(z)=(\alpha+1)(1-|z|^2)^\alpha\,dA(z)$$
denote the weighted area measure on $\D$. The spaces
$$A^p_\alpha=H(\D)\cap L^p(\D,dA_\alpha),\qquad 0<p<\infty,$$
are called weighted Bergman spaces with standard weights.

We will see that Theorem A can be thought of as an extension of the following atomic 
decomposition for weighted Bergman spaces, which can be found in \cite{Z1} for example.

\begin{thm}
Suppose $0<p<\infty$, $\alpha>-1$, and 
\begin{equation}
b>\max(1,1/p)+(\alpha+1)/p.
\label{eq2}
\end{equation}
There exists some positive number $\delta$ such that for any $r$-lattice $\{z_n\}$ with 
$r<\delta$ the weighted Bergman space $A^p_\alpha$ consists exactly of functions of the form
\begin{equation}
f(z)=\sum_{n=1}^\infty c_n\frac{(1-|z_n|^2)^{(pb-2-\alpha)/p}}{(1-z\overline z_n)^b},
\label{eq3}
\end{equation}
where $\{c_n\}\in l^p$.
\label{5}
\end{thm}

Atomic decomposition for Bergman spaces was first obtained in \cite{CR}. We will follow
the proof of the above theorem as found in \cite{Z1}. We begin with an explict construction 
for a measure in (\ref{eq1}) when $f$ is a polynomial.

\begin{lemma}
If $f$ is a polynomial and $0<p\le\infty$, there exists a complex Borel measure $\mu$ 
such that the localized function $|\mu|_r(z)$ is in $L^p(\D,d\lambda)$ and
$$f(z)=\ind\frac{z-w}{1-z\overline w}\,d\mu(w)$$
for all $z\in\D$.
\label{6}
\end{lemma}

\begin{proof}
If $f$ is a nonzero constant function, we use the measure
$$d\mu(w)=c\,\frac{|w|}w\,(1-|w|^2)^N\,dA(w),$$
where $c$ is an appropriate constant and $N$ is sufficiently large (depending on $p$). 
If $f(w)=w^n$ with $n\ge1$, we use the measure
$$d\mu(w)=c\,w^{n-1}(1-|w|^2)^N\,dA(w),$$
where $c$ is an appropriate constant and $N$ is sufficiently large (depending on $p$).
This follows from the Taylor expansion of the function $1/(1-z\overline w)$ and 
polar coordinates.
\end{proof}

We now proceed to the proof of Theorem A. First assume that $f\in B_p$ for some
$0<p\le1$. Let $k$ be any positive integer such that $pk>1$. Let $b=k+1$ and 
$\alpha=pk-2$. Then $b$ satisfies the condition in (\ref{eq2}). In fact, since $0<p\le1$,
we have
$$\max\left(1,\frac1p\right)+\frac{\alpha+1}p=\frac1p+\frac{pk-1}p=k<b.$$
Also, $f\in B_p$ if and only if its $k$-th order derivative $f^{(k)}$ is in $A^p_\alpha$ 
and the exponent in the numerator of (\ref{eq3}) is
$$\frac{pb-2-\alpha}p=1.$$
It follows from Theorem~\ref{5} that we can find an $r$-lattice $\{z_n\}$ in the
pseudo-hyperbolic metric and a sequence $\{c_n\}\in l^p$ such that
$$f^{(k)}(z)=\sum_{n=1}^\infty c_n\frac{1-|z_n|^2}{(1-z\overline z_n)^{k+1}},
\qquad z\in\D.$$
There is considerable freedom in the choice of the $r$-lattice in Theorem~\ref{5}. So
we may assume that $z_n\not=0$ for each $n$ and $|z_n|\to1$ as $n\to\infty$. We then 
consider the function
$$g(z)=\sum_{n=1}^\infty c_n'\frac{z-z_n}{1-z\overline z_n},\qquad z\in\D,$$
where 
$$c_n'=\frac{c_n}{k!\,\overline z_n^{k-1}},\qquad n\ge1.$$
Clearly, the sequence $\{c_n'\}$ is still in $l^p\subset l^1$. This is where we use
the assumption $0<p\le1$ in a critical way to ensure that the infinite series
defining $g$ actually converges. Differentiating term by term shows that the function 
$g$ satisfies $g^{(k)}=f^{(k)}$. Thus there is a polynomial $P(z)$ such that
$$f(z)=P(z)+\sum_{n=1}^\infty c_n'\frac{z-z_n}{1-z\overline z_n}.$$
By Lemma~\ref{6}, there is a measure $\nu$ such that the localized function $|\nu|_r(z)$ 
is in $L^p(\D,d\lambda)$ and
$$P(z)=\ind\frac{z-w}{1-z\overline w}\,d\nu(w),\qquad z\in\D.$$
If we define
$$\mu=\nu+\sum_{n=1}^\infty c_n'\delta_{z_n},$$
where $\delta_{z_n}$ denotes the unit point mass at $z_n$,
we obtain the desired representation for $f$ with the localized function $|\mu|_r(z)$ 
belonging to $L^p(\D,d\lambda)$. This proves one direction of Theorem A.

To prove the other direction of Theorem A, let us assume that $\mu$ is a complex Borel
measure such that the function $|\mu|_r(z)$ is in $L^p(\D,d\lambda)$. Let $r$ be a
sufficiently small radius and $\{z_n\}$ be an $r$-lattice in the pseudo-hyperbolic metric.
By Corollary~\ref{4} the sequence $|\mu|_r(z_n)$ is in $l^p$. Now if
$$f(z)=\ind\frac{z-w}{1-z\overline w}\,d\mu(w),\qquad z\in\D,$$
then
$$f^{(k)}(z)=k!\ind\frac{(1-|w|^2)\,\overline w^{k-1}}{(1-z\overline w)^{k+1}}\,d\mu(w),
\qquad z\in\D.$$
We use Lemma~\ref{2} to decompose $\D$ into the disjoint union of $\{D_n\}$ and rewrite
$$f^{(k)}(z)=k!\sum_{n=1}^\infty\int_{D_n}\frac{(1-|w|^2)\,\overline w^{k-1}}
{(1-z\overline w_n)^{k+1}}\,d\mu(w),$$
so that
$$|f^{(k)}(z)|\le k!\sum_{n=1}^\infty\int_{D_n}\frac{1-|w|^2}{|1-z\overline w|^{k+1}}
\,d|\mu|(w).$$
For each $n\ge1$ and $z\in\D$ there is some point $w_n(z)\in D_n$ such that
$$\frac{1-|w_n(z)|^2}{|1-z\overline w_n(z)|^{k+1}}=\sup_{w\in D_n}
\frac{1-|w|^2}{|1-z\overline w|^{k+1}}.$$
Thus,
$$|f^{(k)}(z)|\le\sum_{n=1}^\infty c_n\frac{1-|w_n(z)|^2}{|1-z\overline w_n(z)|^{k+1}}$$
for all $z\in\D$, where $c_n=k!\,|\mu|(D_n)$ is a sequence in $l^p$ as $|\mu|(D_n)\le
|\mu|(D(z_n,r))$. By Lemma 4.30 of \cite{Z1}, there exists a constant $C>0$ (independent
of $n$ and $z$) such that
$$\frac{1-|w_n(z)|^2}{|1-z\overline w_n(z)|^{k+1}}\le C\,\frac{1-|z_n|^2}
{|1-z\overline z_n|^{k+1}}$$
for all $n\ge1$ and all $z\in\D$. Therefore,
$$|f^{(k)}(z)|\le C\sum_{n=1}^\infty c_n\frac{1-|z_n|^2}{|1-z\overline z_n|^{k+1}}$$
for all $z\in\D$. Since $0<p\le1$, we apply H\"older's inequality to get
$$|f^{(k)}(z)|^p\le C^p\sum_{n=1}|c_n|^p\frac{(1-|z_n|^2)^p}{|1-z\overline z_n|^{p(k+1)}}.$$
Integrate term by term and apply Proposition 1.4.10 of \cite{R}. We obtain another
constant $C>0$ (independent of $f$) such that
$$\ind|f^{(k)}(z)|^p(1-|z|^2)^{pk-2}\,dA(z)\le C\sum_{n=1}^\infty|c_n|^p<\infty,$$
which shows that $f\in B_p$ and completes the proof of Theorem A.

It is clear that Theorem A cannot possibly be true when $p>1$. This is because any function 
$f$ represented by the integral in Theorem A must be bounded, while there are unbounded
functions in $B_p$ when $p>1$.

\section{Bergman type spaces on the unit ball}

In this section we show how to extend Theorem A to Bergman type spaces on
the unit ball $\bn$ in $\cn$. The main reference for this section is \cite{RZ}.
When $\alpha>-1$, all background information can also be found in \cite{Z2}.

For any real parameter $\alpha$ we consider the weighted volume measure 
$$dv_\alpha(z)=(1-|z|^2)^\alpha\,dv(z),$$
where $dv$ is the Lebesgue volume measure on $\bn$. 

For real $\alpha$ and $0<p<\infty$ we use $A^p_\alpha$ to denote the space of holomorphic
functions $f$ in $\bn$ such that $(1-|z|^2)^kR^kf(z)$ is in $L^p(\bn,dv_\alpha)$, where
$k$ is a nonnegative integer satisying $pk+\alpha>-1$ and $Rf$ is the standard radial 
derivative defined by
$$Rf(z)=z_1\frac{\partial f}{\partial z_1}+\cdots+z_n\frac{\partial f}{\partial z_n}.$$
It is well known that the space $A^p_\alpha$ is independent of the integer $k$ used in 
the definition.

Various names exist in the literature for the spaces $A^p_\alpha$: Bergman spaces, Besov
spaces, and Sobolev spaces, among others. We follow \cite{RZ} and call them Bergman spaces 
here. When $\alpha>-1$, $A^p_\alpha$ are indeed the weighted Bergman spaces with standard 
weights. For $\alpha=-(n+1)$, $A^p_\alpha$ become the so-called diagonal Besov spaces.

If $p$ is fixed, all the spaces $A^p_\alpha$ are isomorphic as Banach spaces for 
$1\le p<\infty$ and as complete metric spaces for $0<p<1$. The isometry can be realized
by certain fractional radial differential operators. Because of this, it is often enough 
for us just to consider the case $\alpha=0$ and obtain the other cases by fractional 
differentiation or fractional integration.

On the unit ball there exists a unique family of involutive automorphisms $\varphi_a(z)$
that are high dimensional analogs of the M\"obius maps
$$\varphi_a(z)=\frac{a-z}{1-\overline az}$$
on the unit disk. See \cite{R} and \cite{Z2}. The pseudo-hyperbolic metric on $\bn$ is
still the metric defined by $d(z,w)=|\varphi_z(w)|$. For any complex Borel measure $\mu$ 
on $\bn$ the localized function $\mu_r$ and the averaging
function $\widehat\mu_r$ are defined in exactly the same way as before. 

We can now extend Theorem A to all the spaces $A^p_\alpha$ as follows.

\begin{thm}
Suppose $\alpha$ is real, $0<p<\infty$, $0<r<1$, and
$$b>\max\left(1,\frac1p\right)+\frac{\alpha+1}p.$$
Then a function $f\in H(\bn)$ belongs to $A^p_\alpha$ if and only if
\begin{equation}
f(z)=\inb\frac{(1-|w|^2)^{(pb-n-1-\alpha)/p}}{(1-\langle z,w\rangle)^b}\,d\mu(w)
\label{eq7}
\end{equation}
for some complex Borel measure $\mu$ on $\bn$ with the localized function 
$|\mu|_r(z)$ belonging to $L^p(\D,d\lambda)$, where
$$d\lambda(z)=\frac{dv(z)}{(1-|z|^2)^{n+1}}$$
is the M\"obius invariant volume measure on $\bn$.
\label{7}
\end{thm}

\begin{proof}
That every function $f\in A^p_\alpha$ has the desired integral representation 
in (\ref{eq7}), with $\mu$ being atomic, follows from Theorem~{32} in \cite{RZ}, 
which is the atomic decomposition for these spaces.

On the other hand, if $f$ is a function represented by (\ref{eq7}), we follow the second
half of the proof of Theorem A to obtain the following estimate,
$$|R^Nf(z)|\le C\sum_{k=1}^\infty c_k\frac{(1-|z_k|^2)^{(pb-n-1-\alpha)/p}}{|1-
\langle z,z_k\rangle|^{b+N}},\qquad z\in\bn,$$
where $C$ is some positive constant, $\{c_k\}\in l^p$, and $N$ is a sufficiently large
positive integer. Using the arguments on pages 92-93 of \cite{Z1} (the proof for atomic
decomposition of Bergman spaces) we can then show that $R^Nf$ belongs to the Bergman
space $A^p_{Np+\alpha}$, which means that $f\in A^p_\alpha$.
\end{proof}

One particular case is worth mentioning. If $p=1$ and $\mu$ is a positive Borel
measure on $\bn$, then we use Fubini's theorem, the fact that 
$\chi_{D(z,r)}(w)=\chi_{D(w,r)}(z)$, and the M\"obius invariance of $d\lambda$ to obtain
\begin{eqnarray*}
\inb\mu_r(z)\,d\lambda(z)&=&\inb\frac{\mu(D(z,r))\,dv(z)}{(1-|z|^2)^{n+1}}\\
&=&\inb\frac{dv(z)}{(1-|z|^2)^{n+1}}\inb\chi_{D(z,r)}(w)\,d\mu(w)\\
&=&\inb\,d\mu(w)\inb\frac{\chi_{D(w,r)}(z)\,dv(z)}{(1-|z|^2)^{n+1}}\\
&=&\inb\,d\mu(w)\int_{D(w,r)}\frac{dv(z)}{(1-|z|^2)^{n+1}}\\
&=&\inb\,d\mu(w)\int_{D(0,r)}\frac{dv(z)}{(1-|z|^2)^{n+1}}\\
&=& C_r\,\mu(\bn).
\end{eqnarray*}
Therefore, for general complex Borel measures $\mu$ on $\bn$, the condition 
$|\mu|_r\in L^1(\bn,d\lambda)$ is equivalent to the condition that $|\mu|(\bn)<\infty$. 
This is the original condition for $\mu$ in the integral representation of the minimal 
Besov space $B_1$.

It is also interesting to note that there exist many integral representations of
Bergman type spaces in terms of $L^p$ functions when $p\ge1$. See \cite{RZ} for
example. But there is no integral representation in terms of general $L^p$ 
functions when $p<1$, because the integrals cannot even be properly set up in this case.
But Theorem~\ref{7} above tells us that we can still obtain integral representations
in terms of measures when $0<p<1$. This allows us to recover some special integral
representation formulas in \cite{BB} in terms of certain special functions in the 
case $0<p<1$.

\section{Lipschitz spaces in the unit ball}

For any real number $t$ the holomorphic Lipschitz space $\Lambda_t$ on $\bn$ consists
of functions $f\in H(\bn)$ such that $(1-|z|^2)^{k-t}R^kf(z)$ is bounded, where
$k$ is any nonnegative integer greater than $\alpha$. It is known that the space
$\Lambda_t$ is independent of the integer $k$ used in the definition. See \cite{RZ}
for this and other background information about these spaces.

When $0<t<1$, $\Lambda_t$ consists of functions $f\in H(\bn)$ such that
$$|f(z)-f(w)|\le C|z-w|^t.$$
When $t=0$, we can take $k=1$ in the definition of $\Lambda_0$ and obtain the Bloch space
$\bloch$ of functions $f\in H(\bn)$ such that
$$\sup_{z\in\bn}(1-|z|^2)|Rf(z)|<\infty.$$
When $t=1$, the resulting Lipschitz space $\Lambda_1$ is usually called the Zygmund class.

We can define a family of Carleson measures on the unit ball in exactly the same way as
in the unit disk. Thus for any $t>0$ we say that a positive Borel measure $\mu$ on $\bn$
is a $t$-Carleson measure if for some (or every) $r\in(0,1)$ we have
$$\sup_{z\in\bn}\frac{\mu(D(z,r))}{(1-|z|^2)^t}<\infty.$$
If $\{z_k\}$ is an $r$-lattice in the pseudo-hyperbolic metric of $\bn$, then $\mu$ is
$t$-Carleson if and only if
$$\sup_{n\ge1}\frac{\mu(D(z_n,r))}{(1-|z_n|^2)^t}<\infty.$$
If $t>n$, then $\mu$ is a $t$-Carleson measure if and only if
\begin{equation}
\inb|f(z)|^p\,d\mu(z)\le C\inb|f(z)|^p\,dv_{t-n-1}(z)
\label{eq5}
\end{equation}
for some constant $C=C_p>0$ and all $f\in H(\bn)$.

The following theorem is modeled on the atomic decomposition for holomorphic Lipschitz
spaces. See Theorem 33 in \cite{RZ}.

\begin{thm}
Suppose $t$ and $b$ are two real parameters satisfying
\begin{enumerate}
\item[(a)] $b+t>n$.\
\item[(b)] $b$ is neither $0$ nor a negative integer.
\end{enumerate}
Then a function $f\in H(\bn)$ is in the Lipschitz space $\Lambda_t$ if and only if there 
exists a complex Borel measure $\mu$ such that the localized function $|\mu|_r(z)$ is 
bounded and
\begin{equation}
f(z)=\inb\frac{(1-|w|^2)^{b+t}}{(1-\langle z,w\rangle)^b}\,d\mu(w)
\label{eq4}
\end{equation}
for all $z\in\bn$.
\label{8}
\end{thm}

\begin{proof}
If $f$ admits the representation in (\ref{eq4}), where the localized function $|\mu|_r$ 
is bounded, then differentiating under the integral sign gives
$$|R^kf(z)|\le C\inb\frac{(1-|w|^2)^{b+t}}{|1-\langle z,w\rangle|^{b+k}}\,d|\mu|(w),
\qquad z\in\bn,$$
where $k$ is a nonnegative integer greater than $t$ and $C$ is a positive constant 
independent of $z$. Let
$$d\nu(z)=(1-|z|^2)^{b+t}\,d|\mu|(z).$$
Then the boundedness of the localized function $|\mu|_r(z)$ is equivalent to the measure
$\nu$ being $(b+t)$-Carleson. Since $b+t>n$, it follows from (\ref{eq5}) that there
is another constant $C>0$, independent of $z$, such that
$$|R^kf(z)|\le C\inb\frac{(1-|w|^2)^{b+t-n-1}}{|1-\langle z,w\rangle|^{b+k}}\,dv(w)$$
for all $z\in\bn$. Since $b+t-n-1>-1$, $k>t$, and
$$b+k=n+1+(b+t-n-1)+(k-t),$$
it follows from Proposition 1.4.10 of \cite{R} that there is another constant $C>0$,
indepedent of $z$, such that
$$|R^kf(z)|\le\frac C{(1-|z|^2)^{k-t}},\qquad z\in\bn.$$
This shows that $f\in\Lambda_t$.

On the other hand, if $f\in\Lambda_t$, then by the atomic decomposition theorem for
Lipschitz spaces (see Theorem 33 of \cite{RZ}), there exists an $r$-lattice 
$\{z_k\}$ in the pseudo-hyperbolic metric and a sequence $\{c_k\}\in l^\infty$ such that
$$f(z)=\sum_{k=1}^\infty c_k\frac{(1-|z_k|^2)^{b+t}}{(1-\langle z,z_k\rangle)^b}$$
for all $z\in\bn$. Let
$$\mu=\sum_{k=1}^\infty c_k\delta_{z_k}.$$
Then
$$f(z)=\inb\frac{(1-|w|^2)^{b+t}}{(1-\langle z,w\rangle)^b}\,d\mu(w),\qquad z\in\bn,$$
and the measure $\mu$ has the property that $|\mu|_r(z)$ is a bounded function. This
completes the proof of the theorem.
\end{proof}

We want to write down an equivalent form of the above theorem which may be more useful 
in certain situations. More specifically, if we write $b=n+1+\alpha-t$, then the
condition $b+t>n$ becomes $\alpha>-1$. If we combine the factor $(1-|w|^2)^{b+t}$ with
the measure $\mu$ in Theorem~\ref{8}, we obtain the following equivalent form.

\begin{thm}
Suppose $t$ is real, $\alpha>-1$, and the number $n+1+\alpha-t$ is neither $0$ nor a
negative integer. Then a function $f\in H(\bn)$ belongs to the Lipschitz space
$\Lambda_t$ if and only if we can represent $f$ by
\begin{equation}
f(z)=\inb\frac{d\mu(w)}{(1-\langle z,w\rangle)^{n+1+\alpha-t}},
\label{eq6}
\end{equation}
where $|\mu|$ is $(n+1+\alpha)$-Carleson, namely,
$$\inb|g(z)|^p\,d|\mu|(z)\le C\inb|g(z)|^p\,dv_\alpha(z)$$
for some constant $C=C(p,\alpha)$ and all $g\in H(\bn)$.
\label{9}
\end{thm}

In this equivalent form, the second half of the proof of Theorem~\ref{8} can be replaced
by an argument based on absolutely continuous measures instead of atomic measures (and so
avoiding the use of atomic decomposition). In fact, if $f\in\Lambda_t$, then by
Theorem 17 of \cite{RZ}, there exists a function $g\in L^\infty(\bn)$ such that
$$f(z)=\inb\frac{g(w)\,dv_\alpha(w)}{(1-\langle z,w\rangle)^{n+1+\alpha-t}},\qquad z\in\bn.$$
Setting $d\mu(w)=g(w)\,dv_\alpha(w)$ leads to the representation in (\ref{eq6}) with
$|\mu|$ being an $(n+1+\alpha)$-Carleson measure.

We now use Theorem~\ref{8} to prove Theorem B.

First assume that $t>1$, $|\mu|$ is $t$-Carleson, and
$$f(z)=\ind\frac{z-w}{1-z\overline w}\,d\mu(w),\qquad z\in\D.$$
Then
$$|f^{(k)}(z)|\le\ind\frac{1-|w|^2}{|1-z\overline w|^{k+1}}\,d|\mu|(w),\qquad z\in\D,$$
where $k$ is any positive integer greater than $t$. Let 
$$d\nu(w)=(1-|w|^2)\,d|\mu|(w)$$
Then $\nu$ is $(t+1)$-Carleson, so there exists a positive constant $C$ such that
$$|f^{(k)}(z)|\le C\ind\frac{(1-|w|^2)^{t-1}}{|1-z\overline w|^{k+1}}\,dA(w)$$
for all $z\in\D$. An application of Proposition 1.4.10 of \cite{R} then produces another
positive constant $C$ such that
$$|f^{(k)}(z)|\le\frac C{(1-|z|^2)^{k-t}},\qquad z\in\D,$$
which means that $f\in\Lambda_t$.

Next assume that $t>1$ and $f\in\Lambda_t$. Since
$$(1-|z|^2)^{k-t}f^{(k)}(z)=(1-|z|^2)^{(k-1)-(t-1)}(f')^{(k-1)}(z),$$
where $k$ is sufficiently large, we see that $f'\in\Lambda_{t-1}$. We apply 
Theorem~\ref{8} to the function $f'$ to obtain a complex Borel measure $\nu$ such that
$$f'(z)=\ind\frac{(1-|w|^2)^{t+1}}{(1-z\overline w)^2}\,d\nu(w),\qquad z\in\D,$$
where the localized function $|\nu|_r(z)$ is bounded. Let
$$d\mu(z)=(1-|z|^2)^t\,d\nu(z).$$
Then $|\mu|$ is $t$-Carleson and
$$f'(z)=\ind\frac{1-|w|^2}{(1-z\overline w)^2}\,d\mu(w),\qquad z\in\D.$$
Integrate both sides, note that $\mu$ is a finite measure on $\D$ (any $t$-Carleson measure
is finite when $t>1$), and take care of the integration constant using Lemma~\ref{6}.
We obtain the integral representation for $f$ in Theorem B.

It is clear that the conclusion in Theorem B is false if $t\le1$. In fact, if
$$f(z)=\ind\frac{z-w}{1-z\overline w}\,d\mu(w),\qquad z\in\D,$$
for some measure $\mu$, then $\mu$ must be finite and $f$ must be bounded. Although every
function in $\Lambda_t$ is bounded when $t>0$, not every $t$-Carleson measure is finite
when $t\le1$. So the assumption $t>1$ in Theorem B is best possible.

\section{Fock spaces in $\cn$}

We now consider Fock spaces in $\cn$. Throughout this section we let $\alpha$
denote a {\it positive} weight parameter and define
$$d\lambda_\alpha(z)=c_\alpha\,e^{-\alpha|z|^2}\,dv(z),$$
where $dv$ is volume measure in $\cn$ and $c_\alpha$ is a positive normalizing constant so
that $\lambda_\alpha(\cn)=1$. These are called (weighted) Gaussian measures.

For $0<p\le\infty$ let $F^p_\alpha$ denote the space of entire functions $f$ in $\cn$ such that
the function $f(z)e^{-\alpha|z|^2/2}$ belongs to $L^p(\cn,dv)$. These are called Fock spaces.
Sometimes they are also called Bargmann or Segal-Bargmann spaces.

In place of the pseudo-hyperbolic metric for $\bn$, we use the Euclidean metric in this
context. It is even easier to define the notion of lattices in the Euclidean metric. So we
will not elaborate on such details.

The following result is well known and is usually referred to as the atomic decomposition
for Fock spaces. See \cite{JPR} for the case $1\le p\le\infty$ and \cite{W} for the
case $0<p<1$. See \cite{Z3} for more background information about Fock spaces.

\begin{thm}
For any $0<p\le\infty$ there exists a constant $\delta=\delta(p,\alpha)>0$ with the
following property: if $r\in(0,\delta)$ and $\{z_k\}$ is any $r$-lattice in $\cn$ in the 
Euclidean metric, then an entire function $f$ belongs to $F^p_\alpha$ if and only if
$$f(z)=\sum_{k=1}^\infty c_ke^{\alpha\langle z,z_k\rangle-\frac\alpha2|z_k|^2}$$
for some sequence $\{c_k\}\in l^p$.
\label{10}
\end{thm}

The integral representation of Fock spaces in terms of complex Borel measures then 
takes the following form.

\begin{thm}
Let $0<p\le\infty$, $r>0$, and $f$ be an entire function. Then $f\in F^p_\alpha$ if and 
only if there exists a complex Borel measure $\mu$ on $\cn$ such that the localized function
$|\mu|_r(z)=|\mu|(D(z,r))$ belongs to $L^p(\cn,dv)$ and
$$f(z)=\int_{\cn}e^{\alpha\langle z,w\rangle-\frac\alpha2|w|^2}\,d\mu(w)$$
for all $z\in\cn$.
\label{11}
\end{thm}

\begin{proof}
One direction actually follows from the atomic decomposition when we use atomic measures.
The other direction follows from the proof of the atomic decomposition theorem, as adapted
in the previous sections. We omit the details.
\end{proof}

It is well known that for $1\le p\le\infty$, an entire function $f$ in $\cn$ belongs to
the Fock space $F^p_\alpha$ if and only if there exists a function $g$ such that
the function $g(z)e^{-\alpha|z|^2/2}$ is in $L^p(\cn,dv)$ and
$$f(z)=\int_{\cn} e^{\alpha\langle z,w\rangle}g(w)\,d\lambda_\alpha(w)$$
for all $z\in\cn$. It is clear that Theorem~\ref{11} above is an extension of this
integral representation to the case of measures. 

More interesting to us here is when $0<p<1$. In this case, there is no integral
representation of $F^p_\alpha$ in terms of general $L^p$ functions. But Theorem~\ref{11} 
tells us that we can still do integral representation using measures.

Finally we mention that, just as in the case of Bergman and Besov spaces, the condition
that $|\mu|_r\in L^1(\cn,dv)$ is the same as $|\mu|(\cn)<\infty$. Also, the
condition that $|\mu|_r\in L^\infty(\cn)$ is the same as
$$\int_{\cn}\left|f(z)e^{-\frac\alpha2|z|^2}\right|^p\,d|\mu|(z)\le C_p
\int_{\cn}\left|f(z)e^{-\frac\alpha2|z|^2}\right|^p\,dv(z),$$
so it is reasonable to call $|\mu|$ an $F^p_\alpha$-Carleson measure in this case.

\end{document}